\documentclass[12pt]{amsart}      %{article} was 12pt latex e
\usepackage{amssymb}
\usepackage{eucal}
\usepackage{amsmath}
\usepackage{amscd}
\usepackage[all]{xy}           %xypic macro for latex2.09
\usepackage{txfonts}      %{article} was 12pt latex e

\usepackage{amsfonts,latexsym}
\usepackage{xspace}
\usepackage{hyperref}
\usepackage{float}
\usepackage{color}
\usepackage{colordvi}
\usepackage{multicol}

\newcommand{\nc}{\newcommand}
\newcommand{\delete}[1]{}

\nc{\dfootnote}[1]{{}}          %{{}}
\nc{\ffootnote}[1]{\dfootnote{#1}}

\delete{
\nc{\mfootnote}[1]{{}}        % Use this to suppress footnotes
\nc{\ofootnote}[1]{{}}        % Use this to suppress footnotes
}

\nc{\mfootnote}[1]{\footnote{#1}} % Use this to show footnotes
\nc{\ofootnote}[1]{\footnote{\tiny Older version: #1}} % Use this to show footnotes

\nc{\mlabel}[1]{\label{#1}}  % Use this to suppress names
\nc{\mcite}[1]{\cite{#1}}  % Use this to suppress names
\nc{\mref}[1]{\ref{#1}}  % Use this to suppress names
\nc{\mkeep}[1]{{}}      % Use this to suppress marginpar
\nc{\mbibitem}[1]{\bibitem{#1}} % Use this to show number name
\delete{
\nc{\mcite}[1]{\cite{#1}{{\bf{{\ }(#1)}}}}  % Use this lines to show names
\nc{\mlabel}[1]{\label{#1}  % Use the next two lines to show names
{\hfill \hspace{1cm}{\bf{{\ }\hfill(#1)}}}}
\nc{\mref}[1]{\ref{#1}{{\bf{{\ }(#1)}}}}  % Use this lines to show names
\nc{\mbibitem}[1]{\bibitem[\bf #1]{#1}} % Use this to show name
\nc{\mkeep}[1]{\marginpar{{\bf #1}}} % Use this to show marginpar
}

\newtheorem{theorem}{Theorem}[section]
\newtheorem{prop}[theorem]{Proposition}

\newtheorem{coro}[theorem]{Corollary}
\newtheorem{prop-def}{Proposition-Definition}[section]

\nc{\bond}{\vdash}

\nc{\comp}[1]{\langle #1\rangle} \nc{\spr}{\cdot}
\nc{\disp}[1]{\displaystyle{#1}}
\nc{\sk}[1]{\mrm{sk}(#1)}
\nc{\ve}[1]{\mrm{vec}({#1})}
\nc{\pf}[1]{\check{#1}}
\nc{\Li}{\mrm{Li}}
\nc{\icut}{^!}
\nc{\bin}[2]{ (_{\stackrel{\scs{#1}}{\scs{#2}}})}  %binomial coeff
\nc{\binc}[2]{ \bigg (\!\! \begin{array}{c} \scs{#1}\\
    \scs{#2} \end{array}\!\! \bigg )}  %binomial coeff
\nc{\bbinc}[2]{ \left (\!\! \begin{array}{c} {#1}\\
    {#2} \end{array}\!\! \right )}  %binomial coeff
\nc{\bincc}[2]{  \left ( {\scs{#1} \atop
    \vspace{-.5cm}\scs{#2}} \right )}  %binomial coeff
\nc{\sarray}[2]{\begin{array}{c}#1 \vspace{.1cm}\\ \hline
    \vspace{-.35cm} \\ #2 \end{array}}
\nc{\spair}[2]{\big[\begin{array}{c}\scs{#1} \\ \scs{#2} \end{array} \big]}
\nc{\ppair}[2]{\big\langle\begin{array}{c}\scs{#1} \\ \scs{#2} \end{array} \big\rangle}
\nc{\pfpair}[2]{\big(\begin{array}{c}\scs{#1} \\ \scs{#2} \end{array} \big)}
\nc{\zsg}[1]{\widehat{#1}} \nc{\bs}{\bar{S}}
\nc{\dcup}{\stackrel{\bullet}{\cup}}
\nc{\dbigcup}{\stackrel{\bullet}{\bigcup}} \nc{\la}{\longrightarrow}
\nc{\fe}{\'{e}} \nc{\rar}{\rightarrow} \nc{\dar}{\downarrow}
\nc{\dap}[1]{\downarrow \rlap{$\scriptstyle{#1}$}}
\nc{\uap}[1]{\uparrow \rlap{$\scriptstyle{#1}$}}
\nc{\dt}[1]{{#1}^\sharp} \nc{\st}[1]{{#1}^\flat}
\nc{\defeq}{\stackrel{\rm def}{=}} \nc{\dis}[1]{\displaystyle{#1}}
\nc{\dotcup}{\ \displaystyle{\bigcup^\bullet}\ } \nc{\hcm}{\
\hat{,}\ } \nc{\hcirc}{\hat{\circ}} \nc{\hts}{\hat{\shpr}}
\nc{\lts}{\stackrel{\leftarrow}{\shpr}}
\nc{\rts}{\stackrel{\rightarrow}{\shpr}} \nc{\lleft}{[}
\nc{\lright}{]} \nc{\uni}[1]{\tilde{#1}} \nc{\free}[1]{\bar{#1}}
\nc{\den}[1]{\check{#1}} \nc{\lrpa}{\wr} \nc{\curlyl}{\left \{
\begin{array}{c} {} \\ {} \end{array}
    \right . \!\!\!\!\!\!\!}
\nc{\curlyr}{ \!\!\!\!\!\!\!
    \left . \begin{array}{c} {} \\ {} \end{array}
    \right \} }
\nc{\longmid}{\left | \begin{array}{c} {} \\ {} \end{array}
    \right . \!\!\!\!\!\!\!}
\nc{\ot}{\otimes} \nc{\bigot}{\bigotimes} \nc{\mdiv}{\mrm{div}}
\nc{\shpf}{\frakp}
\nc{\sg}{G}
\nc{\tg}{T}
\nc{\ssg}[1]{\overline{#1}}
\nc{\psg}[1]{\widetilde{#1}}
\nc{\zg}{{Z}} \nc{\ig}{I} \nc{\pg}{P} \nc{\jg}{J}
\nc{\eg}{E} \nc{\fg}{F} \nc{\cg}{C} \nc{\mg}{M} \nc{\abg}{C} \nc{\ug}{{U}}
\nc{\bas}{B} \nc{\Lyn}{\mrm{Lyn}} \nc{\lyn}{\Lyn} \nc{\sG}{{\cals}}
\nc{\iG}{{\cali}} \nc{\jG}{{\calj}} \nc{\eG}{{\cale}}
\nc{\pG}{{\calp}} \nc{\fG}{{\calf}} \nc{\cG}{{\calc}}
\nc{\mG}{{\calm}} \nc{\ora}[1]{\stackrel{#1}{\rar}}
\nc{\ola}[1]{\stackrel{#1}{\la}}%${\Bbb Z}$
\nc{\pex}[1]{\{#1\}} \nc{\scs}[1]{\scriptstyle{#1}}
\nc{\mrm}[1]{{\rm #1}} \nc{\sym}[1]{{\widehat{#1}}}
\nc{\margin}[1]{\marginpar{\rm #1}}   %{\rm #1}}
\nc{\dirlim}{\displaystyle{\lim_{\longrightarrow}}\,}
\nc{\invlim}{\displaystyle{\lim_{\longleftarrow}}\,}
\nc{\mvp}{\vspace{0.5cm}} \nc{\svp}{\vspace{2cm}}
\nc{\vp}{\vspace{8cm}} \nc{\proofbegin}{\noindent{\bf Proof: }}
\nc{\proofend}{$\blacksquare$ \vspace{0.5cm}}
\nc{\shqs}{\eta} \font\cyr=wncyr10
\nc{\sha}{{\mbox{\cyr X}}}  %used to be \cyr
\newfont{\scyr}{wncyr10 scaled 550}
\nc{\ssha}{\,\mbox{\bf \scyr X}\,}
\newfont{\bcyr}{wncyr10 scaled 1000}
\nc{\qssha}{{{\ssha\hspace{-2pt}_\ast}}\,}
\nc{\qsshab}{{{\ssha\hspace{-2pt}_{\rho}}}\,}
\nc{\pssha}{{\star}}
\nc{\bsh}{{^{\qsshab}}}
\nc{\ncsha}{{\mbox{\cyr X}^{\mathrm NC}}} \nc{\ncshao}{{\mbox{\cyr
X}^{\mathrm NC,\,0}}}
\nc{\shpr}{\diamond}    %Shuffle product

\nc{\shf}{{^{\ssha}}} \nc{\qsh}{{^{\ast}}}
\nc{\psh}{{^{\pi}}}
\nc{\pfsha}{{\ssha_{\pi}}}
\nc{\esh}{{^{\ssha_\eta}}}
\nc{\lshf}{_{\ssha}} \nc{\lqsh}{_{\ast}} \nc{\lzero}{_{\hskip -5pt
0}} \nc{\shzero}{_{\hskip -7.5pt 0}} \nc{\lone}{_{\hskip -7.5pt 1}}
\nc{\shprl}{{{\shpr}_\lambda}}
\nc{\shpro}{\diamond^0}    %Shuffle product
\nc{\shpru}{\check{\diamond}} \nc{\catpr}{\diamond_l}
\nc{\rcatpr}{\diamond_r} \nc{\lapr}{\diamond_a}
\nc{\lepr}{\diamond_e} \nc{\tcon}{^{\ot}} \nc{\conv}{_c}
\nc{\vep}{\varepsilon} \nc{\labs}{\mid\!} \nc{\rabs}{\!\mid}
\nc{\hsha}{\widehat{\sha}} \nc{\lsha}{\stackrel{\leftarrow}{\sha}}
\nc{\rsha}{\stackrel{\rightarrow}{\sha}}
\nc{\EDS}{{\mrm{EDS}}\xspace} \nc{\DS}{{\mathbf{DS}}}
\nc{\lc}{[} \nc{\rc}{]} \nc{\rbset}{R} \nc{\rbnum}{r}
\nc{\rbfun}{\mathbf{R}} \nc{\pset}{P} \nc{\pnum}{p}
\nc{\pfun}{\mathbf{P}} \nc{\spset}{SP} \nc{\spnum}{sp}
\nc{\spgen}{\mathbf{SP}} \nc{\srbi}[1]{\{#1\}}

\nc{\rind}{r} \nc{\sind}{s} \nc{\tind}{t}  \nc{\kdim}{k}
\nc{\ldim}{\ell}
\nc{\funcf}{{\mathfrak f}}

\nc{\alga}{{A}} \nc{\ann}{\mrm{ann}} \nc{\Aut}{\mrm{Aut}}
\nc{\can}{\mrm{can}} \nc{\colim}{\mrm{colim}} \nc{\Cont}{\mrm{Cont}}
\nc{\rchar}{\mrm{char}} \nc{\cok}{\mrm{coker}} \nc{\dtf}{{R-{\rm
tf}}} \nc{\dtor}{{R-{\rm tor}}}

\nc{\Div}{{\mrm Div}} \nc{\End}{\mrm{End}} \nc{\Ext}{\mrm{Ext}}
\nc{\Fil}{\mrm{Fil}} \nc{\Frob}{\mrm{Frob}} \nc{\Gal}{\mrm{Gal}}
\nc{\GL}{\mrm{GL}} \nc{\lord}{\mrm{L-order}\xspace}
\nc{\rme}{\mrm{E}} \nc{\rmt}{\mrm{T}} \nc{\Sym}{\mrm{Sym}}
\nc{\Hom}{\mrm{Hom}} \nc{\hsr}{\mrm{H}} \nc{\hpol}{\mrm{HP}}
\nc{\id}{\mrm{id}} \nc{\im}{\mrm{im}} \nc{\incl}{\mrm{incl}}
\nc{\length}{\mrm{length}} \nc{\leng}{\mrm{\ell}} \nc{\LR}{\mrm{LR}}
\nc{\mchar}{\mrm char}
\nc{\MZV}{\mrm{MZV}\xspace}
\nc{\MZVs}{\mrm{MZVs}\xspace}
\nc{\MZVf}{\mrm{MZV fractions}\xspace}
\nc{\PF}{\mathbf{PF}}
\nc{\MPV}{\mrm{MPV}\xspace}
\nc{\MPVs}{\mrm{MPVs}\xspace}
\nc{\MPL}{\mrm{MPL}\xspace}
\nc{\MPLs}{\mrm{MPLs}\xspace}
\nc{\mzvalg}{\mathbf{MZV}}
\nc{\mplalg}{\mathbf{MPV}}
\nc{\sfalg}{\mrm{SHF}}
\nc{\qfalg}{\mrm{QHF}}
\nc{\pfalg}{\mathbf{PF}}
\nc{\edsalg}{\mathbf{EDS}} \nc{\qeds}{$\QQ$-EDS\xspace}
\nc{\zeds}{$\ZZ$-EDS\xspace} \nc{\zpeds}{$\ZZ_p$-EDS\xspace}
\nc{\fpeds}{$\FF_p$-EDS\xspace} \nc{\NC}{\mrm{NC}}
\nc{\mpart}{\mrm{part}} \nc{\os}{\mrm{OS}} \nc{\qs}{\mrm{QS}}
\nc{\ql}{{\QQ_\ell}} \nc{\qp}{{\QQ_p}} \nc{\rank}{\mrm{rank}}
\nc{\rcot}{\mrm{cot}} \nc{\rdef}{\mrm{def}} \nc{\rdiv}{{\rm div}}
\nc{\rtf}{{\rm tf}} \nc{\rtor}{{\rm tor}} \nc{\res}{\mrm{res}}
\nc{\sh}{\mrm{Sh}} \nc{\TL}{\mrm{TL}} \nc{\Spec}{\mrm{Spec}}
\nc{\tor}{\mrm{tor}} \nc{\Tr}{\mrm{Tr}} \nc{\tr}{\mrm{tr}}
\nc{\ETC}{\mathrm{ETC}} \nc{\ETL}{\mathrm{ETL}}
\nc{\EL}{\mathrm{EL}} \nc{\RETL}{\mathrm{RETL}}
\nc{\EETL}{\widetilde{\TL}} \nc{\word}{\rm word\xspace}
\nc{\words}{\rm words\xspace} \nc{\varab}{\phi_{\alpha,\beta}}

\nc{\ab}{\mathbf{Ab}} \nc{\Alg}{\mathbf{Alg}}
\nc{\Algo}{\mathbf{Alg}^0} \nc{\Bax}{\mathbf{Bax}}
\nc{\Baxo}{\mathbf{Bax}^0} \nc{\RBo}{\mathbf{RB}^0}
\nc{\BRB}{\mathbf{RB}} \nc{\Dend}{\mathbf{DD}} \nc{\bfe}{{\bf e}}
\nc{\bff}{{\bf f}} \nc{\bfk}{{\bf k}} \nc{\bfone}{{\bf 1}}
\nc{\base}[1]{{a_{#1}}} \nc{\detail}{\marginpar{\bf More detail}
    \noindent{\bf Need more detail!}
    \svp}
\nc{\Diff}{\mathbf{Diff}} \nc{\gap}{\marginpar{\bf
Incomplete}\noindent{\bf Incomplete!!}
    \svp}
\nc{\FMod}{\mathbf{FMod}} \nc{\RB}{\mathbf{RB}}
\nc{\Int}{\mathbf{Int}} \nc{\Mon}{\mathbf{Mon}}
\nc{\remarks}{\noindent{\bf Remarks: }} \nc{\Rep}{\mathbf{Rep}}
\nc{\Rings}{\mathbf{Rings}} \nc{\Sets}{\mathbf{Sets}}
\nc{\DT}{\mathbf{DT}}

\nc{\BA}{{\Bbb A}} \nc{\CC}{{\Bbb C}} \nc{\DD}{{\Bbb D}}
\nc{\EE}{{\Bbb E}} \nc{\FF}{{\Bbb F}} \nc{\GG}{{\Bbb G}}
\nc{\HH}{{\Bbb H}} \nc{\LL}{{\Bbb L}} \nc{\NN}{{\Bbb N}}
\nc{\QQ}{{\Bbb Q}} \nc{\RR}{{\Bbb R}} \nc{\TT}{{\Bbb T}}
\nc{\VV}{{\Bbb V}} \nc{\ZZ}{{\Bbb Z}}

\nc{\cala}{{\mathcal A}} \nc{\calc}{{\mathcal C}}
\nc{\cald}{{\mathcal D}} \nc{\cale}{{\mathcal E}}
\nc{\calf}{{\mathcal F}} \nc{\calg}{{\mathcal G}}
\nc{\calh}{{\mathcal H}} \nc{\cali}{{\mathcal I}}
\nc{\calj}{{\mathcal J}} \nc{\call}{{\mathcal L}}
\nc{\calm}{{\mathcal M}} \nc{\caln}{{\mathcal N}}
\nc{\calo}{{\mathcal O}} \nc{\calp}{{\mathcal P}}
\nc{\calr}{{\mathcal R}} \nc{\calt}{{\mathcal T}}
\nc{\calw}{{\mathcal W}} \nc{\calx}{{\mathcal X}}
\nc{\CA}{\mathcal{A}}

\nc\indI{\mathcal{I}}

\nc{\fraka}{{\mathfrak a}} \nc{\frakB}{{\mathfrak B}}
\nc{\frakb}{{\mathfrak b}} \nc{\frakd}{{\mathfrak d}}
\nc{\frakF}{{\mathfrak F}} \nc{\frakf}{{\mathfrak f}}
\nc{\frakg}{{\mathfrak g}} \nc{\frakL}{{\mathfrak L}}
\nc{\frakm}{{\mathfrak m}} \nc{\frakM}{{\mathfrak M}}
\nc{\frakMo}{{\mathfrak M}^0} \nc{\frakp}{{\mathfrak p}}
\nc{\frakt}{{\mathfrak t}}
\nc{\frakw}{{\mathfrak w}}
\nc{\frakx}{{\mathfrak x}} \nc{\ox}{\overline{\frakx}}
\nc{\frakX}{{\mathfrak X}} \nc{\fraky}{{\mathfrak y}}

\renewcommand\geq{\geqslant}
\renewcommand\leq{\leqslant}

\nc\rbop{{\lc\,\,\rc}} \nc\rbopi[1]{{\lc_{#1} \, \rc_{#1}}}

\nc{\redtext}[1]{\textcolor{red}{#1}}

\begin{document}

\title[Shuffle relation of fractions]
{The shuffle relation of fractions from multiple zeta values}
%
%========================================================================================%
\author{Li Guo}
\address{Department of Mathematics and Computer Science,
         Rutgers University,
         Newark, NJ 07102, USA}
\email{liguo@newark.rutgers.edu}
\author{Bingyong Xie}
\address{Department of Mathematics, East China Normal University, Shanghai, 200241, China}
\email{byxie@math.ecnu.edu.cn}

%=================================================================================
%\date{\today}
%==================================================================================
%\begin{document}

\begin{abstract}
Partial fraction methods play an important role in the study of
multiple zeta values. One class of such fractions is related to the
integral representations of \MZVs. We show that this class  of
fractions has a natural structure of shuffle algebra. This finding
conceptualizes the connections among the various methods of stuffle,
shuffle and partial fractions in the study of \MZVs. This approach
also gives an explicit product formula of the fractions.
\end{abstract}

\maketitle

%==================================================================================

\setcounter{section}{0}

%================================================================================

\section{Introduction}

Let $k$ be a positive integer. For positive integers $s_i$ and variables $u_i$, $1\leq i\leq k$, define
\begin{equation} \funcf\pfpair{s_1,\cdots,s_k}{u_1,\cdots,u_k}
:= \frac{1}{(u_1+\cdots+u_k)^{s_1}(u_2+\cdots+u_k)^{s_2}\cdots
u_k^{s_k}}. \mlabel{eq:pafr}
\end{equation}
In the spacial case when $s_i=1$, $1\leq i\leq k,$ such fractions appeared in connection with differential geometry~\mcite{Bu,DN} and
polylogarithms~\mcite{Go3} where their products were shown to satisfy
the shuffle relation. For example,
\begin{eqnarray} \frac{1}{u_1}\frac{1}{(v_1+v_2)v_2}&=&\frac{1}{(u_1+v_1+v_2)(v_1+v_2)v_2} +\frac{1}{(u_1+v_1+v_2)(u_1+v_2)v_2}\notag \\
&& + \frac{1}{(u_1+v_1+v_2)(u_1+v_2)v_2}.
\mlabel{eq:simplesh}
\end{eqnarray}
In general, such fractions occur naturally
from multiple zeta values which, since their introduction in the
early 1990s, have attracted much attention from a wide range of
areas in mathematics and mathematical physics~\mcite{3BL2,B-K,GKZ,G-M,GZ,Ho1,IKZ,Ko,M-P,Te}. Multiple zeta values (\MZVs) are special
values of the multi-variable complex functions
$$ \zeta(s_1,\cdots,s_k)=\sum_{n_1>\cdots >n_s>0} \frac{1}{n_1^{s_1}\cdots n_k^{s_k}}$$
at positive integers $s_i, 1\leq i\leq k$ with $s_1\geq 2$.
With the change of variable $n_i=u_i+\cdots+u_k, 1\leq i\leq k,$ we have the
well-known rational fraction representation of multiple zeta values:
\begin{equation}
\zeta(s_1,\cdots,s_k)=\sum_{u_1,\cdots,u_k\geq 1}
\funcf\pfpair{s_1,\cdots,s_k}{u_1,\cdots,u_k}. \mlabel{eq:mzvpf}
\end{equation}
For this reason, we will call these fractions
$\funcf\pfpair{s_1,\cdots,s_k}{u_1,\cdots,u_k} $ the {\bf MZV
fractions}. Thus any relation among \MZV  fractions gives a
corresponding relation  among \MZVs after summing over the indices
$u_i$'s. Indeed, many relations among multiple zeta values are
obtained by studying relations among \MZV fractions. The method, the
so called partial fractional method, can be traced back to Euler in
the case when $k=2$ and remains one of the most effective methods
until today~\mcite{GKZ,Gr,Ohno}. For example, Granville~\mcite{Gr} used this method to give a proof of the sum formula. Gangl, Kaneko and Zagier~\mcite{GKZ} used the fraction formula
\begin{equation}
\frac{1}{m^i}\, \frac{1}{n^j}= \sum_{r+s=i+j}
\left ( \binc{r-1}{i-1}\frac{1}{(m+n)^rn^s}
+\binc{r-1}{j-1}\frac{1}{(m+n)^rm^s}\right),
\mlabel{eq:prod2}
\end{equation}
to obtain the well-known Euler's decomposition formula
$$ \zeta(i)\zeta(j)= \sum_{r+s=i+j}
\left ( \binc{r-1}{i-1}\zeta(r,s)
+\binc{r-1}{j-1}\zeta(r,s)\right), i,j\geq 2.
$$
This formula of Euler has been generalized recently by the authors~\mcite{GX2} to a product formula of any two MZVs.

The study of these fractions are interesting on their own right because of their applications outside of \MZVs and that they make sense even if $s_1=1$ when $\zeta(s_1,\cdots,s_k)$ is no longer defined. For example when $s_i=1$ for all $1\leq i\leq k$, these fractions are shown to multiply according to the shuffle product rule~\mcite{Bu,Go3} as mentioned above. However, a product formula for
two \MZV fractions is known only in this case and in the case of Eq.~(\mref{eq:prod2}). In this paper we will provide a product formula for any two \MZV fractions making use of the general double shuffle framework introduced in our previous work~\mcite{GX2} which is obtained with motivation from the shuffle relation and quasi-shuffle (stuffle) relation of \MZVs.
We are able to apply this general framework by showing that the \MZV fractions in Eq.~(\mref{eq:pafr}) have canonical integral representations, in the spirit of the integral representations of \MZVs by Konsevich~\mcite{Ko}. By the standard summation process for \MZVs, we recover the above mentioned generalization of Euler's decomposition formula of \MZVs.

As an example of our product formula, we have
\begin{eqnarray}
\frac{1}{u_1^{r_1}}\, \frac{1}{(v_1+v_2)^{s_1}v_2^{s_2}}  &=
&
\hspace{-.5cm}\sum\limits_{\tiny
\begin{array}{l}\tind_1,\tind_2,\tind_3\geq 1
\\ \tind_1+\tind_2
=\rind_1+\sind_1 \end{array}} \binc{\tind_1-1}{\rind_1-1}
\frac{1}{(u_1+v_1+v_2)^{t_1}(v_1+v_2)^{t_2}v_2^{s_2}}
\notag
\\ && +
\sum_{\tiny
\begin{array}{l}\tind_1,\tind_2,\tind_3\geq 1
\\ \tind_1+\tind_2+\tind_3 \\
=\rind_1+\sind_1+\sind_2 \end{array}} \bigg[
\binc{\tind_1-1}{\sind_1-1}\binc{\tind_2-1}{s_2-t_3}
\frac{1}{(u_1+v_1+v_2)^{t_1}(u_1+v_2)^{t_2}v_2^{t_3}}
\label{eq:1-2}\\
    &&  \qquad \qquad \qquad +
\binc{\tind_1-1}{\sind_1-1}\binc{\tind_2-1}{\sind_2-1}
\frac{1}{(u_1+v_1+v_2)^{t_1}(v_2+u_1)^{t_2}u_1^{t_3}}\bigg].
\notag
\end{eqnarray}
When $r_1=s_1=s_2=1$, we get Eq.~(\mref{eq:simplesh}).
See Theorem~\mref{thm:pfprod} for the general formula.

In Section~\mref{sec:shfr}, we recall our general framework of double shuffle algebras and show that it encodes the shuffle product of \MZV fractions (Theorem~\mref{thm:pf}) through their integral representations. The explicit product formula of \MZV fractions is given in Section~\mref{sec:expf} where we also give some examples.

\section{The algebra of \MZV fractions}
\mlabel{sec:shfr}
In this section, we recall the general double shuffle framework in~\mcite{GX2} and apply it to give the shuffle product of \MZV fractions.

\subsection{Shuffle product of \MZV fractions} \mlabel{ss:pf}
Let $\ug$ be a set. Define the set of symbols
$$ \zsg{\ug}:= \{ \spair{r}{u}\ |\ r\in \ZZ_{\geq 1}, u\in U\}.$$
Let $M(\zsg{\ug})$ be the free monoid generated by $\zsg{\ug}$. Define the free abelian group
\begin{equation}
\calh(\zsg{\ug}): = \ZZ M(\zsg{\ug}). \mlabel{eq:pfalg}
\end{equation}
We will define a product on $\calh(\zsg{\ug})$ by transporting the shuffle product on another algebra.

Define the set of symbols
$$\ssg{\ug}=\{x_0\}\sqcup \{ x_u\ |\ u\in \ug\}$$
and let $M(\ssg{\ug})$ be the free monoid on $\ssg{\ug}$.
As usual~\mcite{GX2,Re}, define the shuffle algebra on $\ssg{\ug}$ to be the vector space
$$ \calh\shf(\ssg{\ug}): = \ZZ M(\ssg{\ug})$$
equipped with the shuffle product $\ssha$, namely
$$
(\alpha_1\vec{\alpha}') \ssha (\beta_1 \vec{\beta}')
= \alpha_1 (\vec{\alpha}' \ssha (\beta_1 \vec{\beta}'))
+ \beta_1 ((\alpha_1\vec{\alpha}')\ssha \vec{\beta}'),
\quad \alpha_1,\beta_1\in \ssg{\ug}, \vec{\alpha}',\vec{\beta}'\in M(\ssg{\ug}),
$$
with the initial condition $1\ssha \vec{\alpha}=\vec{\alpha} = \vec{\alpha} \ssha 1.$

Define the subalgebra
$$ \calh\shf\lone(\ssg{\ug}) = \ZZ \oplus (\oplus_{u\in \ug} \calh\shf(\ssg{\ug})x_u).
$$
Define a linear bijection
\begin{equation}
\rho: \calh\shf\lone(\ssg{\ug}) \to \calh(\zsg{\ug}), \quad
x_0^{r_1-1}x_{u_1}\cdots x_0^{r_k-1}x_{u_k} \mapsto
\spair{r_1,\cdots,r_k}{u_1,\cdots, u_k}. \mlabel{eq:shpfa}
\end{equation}
We then transport the shuffle product $\ssha$ on $\calh\shf(\ssg{\ug})$
to a product $\qsshab$ on $\calh(\zsg{\ug})$ via $\rho$, namely
\begin{equation}
\alpha \qsshab \beta = \rho(\rho^{-1}(\alpha) \ssha \rho^{-1}(\beta)).
\mlabel{eq:pfsh}
\end{equation}
Let $\calh\bsh(\zsg{\ug})$ denote the resulting algebra
$(\calh(\zsg{\ug}), \qsshab).$

Now let $\ug$ be a set
of variables and let $\ZZ(U)$ be the field of rational
functions in $U$. In other words, $\ZZ(U)$ is the field of fractions of $\ZZ[U]$.
Consider the $\ZZ$-submodule
\begin{equation}
 \PF(U):= \ZZ\{ \frakf\pfpair{s_1,\cdots,s_k}{u_1,\cdots,u_k}\ |\ s_i\geq 1, u_i\in U, 1\leq i\leq k, k\geq 0\},
\mlabel{eq:pfsp}
\end{equation}
where $\frakf\pfpair{s_1,\cdots,s_k}{u_1,\cdots,u_k}$ is defined in Eq.~(\mref{eq:pafr}).
The main result of this section is the following

\begin{theorem}
If $\ug$ be a set of variables, then the $\ZZ$-linear map
$$\calf: \calh(\zsg{\ug})\rightarrow \ZZ(U), \quad \calf\spair{\vec{s}}{\vec{u}}= \frakf\pfpair{\vec{s}}{\vec{u}}, \ \calf(1) = 1 $$
is a $\ZZ$-algebra homomorphism.
In particular, the $\ZZ$-submodule $\PF(U)$ of $\ZZ(U)$, as the image of $\calf$, is a $\ZZ$-subalgebra of $\ZZ(U)$.
\mlabel{thm:pf}
\end{theorem}
The proof of this theorem will be given in Section~\mref{sec:intpf}. We first give a consequence of the theorem.

\begin{coro}
The multiplication of two \MZV fractions in $\PF(U)$ satisfies the shuffle relation:
\begin{equation} \mlabel{eq:sh-frac}
\frakf\pfpair{\vec{r}}{\vec{u}} \frakf\pfpair{\vec{s}}{\vec{v}}
= \frakf\big(\pfpair{\vec{r}}{\vec{u}}\qsshab \pfpair{\vec{s}}{\vec{v}} \big).
\end{equation}
Here $\qsshab$ is as defined in Eq.~(\mref{eq:pfsh}).
\mlabel{co:pfshf}
\end{coro}

\subsection{Integral representations of \MZV fractions}
\mlabel{sec:intpf} In Section~\mref{ss:intpf} we give integral
representations of \MZV fractions. We then use this integral representation to prove
Theorem~\mref{thm:pf}.

\subsubsection{Integral representation of \MZV fractions}
\mlabel{ss:intpf}
In preparation of our proof of Theorem~\mref{thm:pf}, we present an integral representation of \MZV fractions which is essentially the same as the well-known integral representation of \MZVs by Konsevich~\mcite{Ko}. For the sake of being self-contained and for later reference, we provide the notations and some details.

Define
\begin{equation}
\alga:=\RR \{ e^{bt} \ |\ b\geq 0 \}, \quad
\alga^+=\RR \{ e^{bt} \ | \ b >0 \}.
\mlabel{eq:A}
\end{equation}
Then $\alga$ and $\alga^+$ are closed under function multiplication and $\alga=\RR \oplus \alga^+$. We define the operator
$$I_0: \alga^+\rightarrow \alga, \quad f(t)\mapsto \int_{-\infty}^t
f(t_1)dt_1.$$ For any $\lambda>0$ we define the operator
$$ I_\lambda: \alga\rightarrow \alga , \quad
f(t) \mapsto \int_{-\infty} ^t f(t_1)e^{\lambda t_1}dt_1.$$ Then we
have the equations:
\begin{equation}
I_0(e^{b t})= \frac{1}{b}e^{b t}, \hskip 10pt b >0, \mlabel{eq:intI0}
\end{equation}
\begin{equation}
I_\lambda(e^{b t})= \frac{1}{b+\lambda} e^{(b+\lambda)t}, \hskip
10pt b\geq 0, \lambda > 0. \mlabel{eq:intIlam}
\end{equation} So
$I_0(\alga^+)\subseteq \alga^+$ and $I_\lambda(\alga)\subseteq
\alga^+$ for $\lambda >0$. By a direct computation using Eq.~(\mref{eq:intI0}) and (\mref{eq:intIlam}) we obtain
\begin{equation}
I_{\lambda_1}(h_1)I_{\lambda_2}(h_2)
=I_{\lambda_1}(h_1I_{\lambda_2}(h_2))
+I_{\lambda_2}(I_{\lambda_1}(h_1)h_2), \label{eq:intI}
\end{equation}
where  $\lambda_1,\lambda_2\in \RR_{\geq 0}$, $h_1$ is  in the
domain of $I_{\lambda_1}$ and $h_2$ is in the domain of
$I_{\lambda_2}$,

\begin{prop}
For any $\frakf\pfpair{s_1,\cdots,s_k}{b_1,\cdots,b_k}\in\RR$, we
have the integral representation
\begin{equation}
\frakf\pfpair{s_1,\cdots,s_k}{b_1,\cdots,b_k}e^{(b_1+\cdots+b_k)t} =
(I_0^{\circ (s_1-1)}\circ I_{b_1}\circ \cdots \circ I_0^{\circ
(s_k-1)} \circ I_{b_k})(1), \mlabel{eq:frintf}
\end{equation}
where $1:\RR\to \RR$ is the constant function.
In particular,
\begin{equation}
\frakf\pfpair{s_1,\cdots,s_k}{b_1,\cdots,b_k} = (I_0^{\circ
(s_1-1)}\circ I_{b_1}\circ\cdots\circ I_0^{\circ (s_k-1)}\circ
I_{b_k})(1)\big|_{t=0}. \mlabel{eq:frint}
\end{equation}
\mlabel{pp:frint}
\end{prop}
\begin{proof} We only need to prove Eq.~(\mref{eq:frintf}) for which we use the induction on
$|\vec{s}|=s_1+\cdots+s_k$. If $|\vec{s}|=1$, then $k=1$ and
$s_1=1$. By Eq.~(\mref{eq:intIlam}) the right hand side of
Eq.~(\mref{eq:frintf}) is $I_{b_1}(1)=\frac{e^{b_1t}}{u_1}$, which
is equal to the left hand side. Let $n$ be a positive integer $\geq
2$. Assume that Eq.~(\mref{eq:frintf}) holds for any $\vec{s}$ with
$|\vec{s}|<n$. Now assume that $|\vec{s}|=n$. If $s_1=1$, then
$k\geq 2$. In this case by the induction hypothesis and Eq.~(\mref{eq:intIlam}) the
right hand side of Eq.~(\mref{eq:frintf}) is equal to
$$I_{b_1}(\frakf\pfpair{s_2,\cdots, s_k}{b_2,\cdots,
b_k}e^{(b_2+\cdots+b_k)t})=
\frakf\pfpair{s_2,\cdots, s_k}{b_2,\cdots,
b_k} I_{b_1}(e^{(b_2+\cdots+b_k)t})=
\frac{1}{b_1+\cdots+
b_k}\frakf\pfpair{s_2,\cdots, s_k}{b_2,\cdots,
b_k}e^{(b_1+\cdots+b_k)t}$$ which coincides with the left hand side.
The argument for $s_1>1$ is similar by using Eq.~(\mref{eq:intI0})
instead of Eq.~(\mref{eq:intIlam}).
\end{proof}

\subsubsection{The proof of Theorem~\mref{thm:pf}}

We now take $U=\RR_+$ in Section~\mref{ss:pf} and define the set $\zsg{\RR_+}$ and the algebra $\calh(\zsg{\RR_+})$.

\begin{prop} The $\RR$-linear map
$$\Theta: \RR\ot _\ZZ \calh(\zsg{\RR_+})\rightarrow \alga, \quad \spair{s_1,\cdots, s_k}{b_1,\cdots, b_k}
\mapsto \frakf\pfpair{s_1,\cdots, s_k}{b_1,\cdots,
b_k}e^{(b_1+\cdots +b_k)t}, \ 1 \mapsto 1
$$ is an $\RR$-algebra homomorphism. \mlabel{pp:realfun}
\end{prop}
\begin{proof}
Define
$$
\begin{aligned}
P_0:& \calh^+(\zsg{\RR_+})\rightarrow \calh(\zsg{\RR_+}), \quad
P_0(\spair{s_1,s_2,\cdots ,s_k}{b_1,b_2,\cdots,b_k})
= \spair{s_1+1, s_2, \cdots, s_k}{\;\;\;\;\;b_1,b_2,\cdots,b_k}, \\
P_b:& \calh(\zsg{\RR_+})\rightarrow \calh(\zsg{\RR_+}), \quad
P_b(\spair{s_1,\cdots, s_k}{b_1,\cdots, b_k})= \spair{ 1,
 s_1, \cdots, s_k}{b, b_1, \cdots, b_k }, \quad
P_b(1)=\spair{1}{b}
\end{aligned}
$$
and take their scalar extensions to $\RR$.
We show that
\begin{equation}
\Theta\circ P_b=I_b\circ \Theta, \quad b\geq 0. \mlabel{eq:comm}
\end{equation}
For $b=0$ we have
$$\begin{aligned}
& \Theta\circ P_0(\spair{s_1,s_2\cdots, s_k}{b_1,b_2\cdots, b_k})=
\Theta(\spair{s_1+1,s_2,\cdots, s_k}{b_1,b_2,\cdots, b_k}) =
\frakf\pfpair{s_1+1,s_2,\cdots, s_k}{b_1, b_2, \cdots,
b_k}e^{(1+s_1+\cdots+s_k)t}
\\
& = I_0\Big((I_0^{\circ s_1-1}\circ I_{b_1}\circ \cdots \circ
I_0^{\circ (s_k-1)} \circ I_{b_k})(1)\Big) =
I_0(\Theta(\spair{s_1,s_2\cdots, s_k}{b_1,b_2\cdots, b_k})),
\end{aligned}$$
where we have used Eq.~(\mref{eq:frintf}) in the last two equations. The argument for $b>0$ is similar.

From \cite[Proposition~4.3]{GX2} we obtain
\begin{equation}
P_a(\xi_1)\qsshab P_b(\xi_2)=P_a(\xi_1\qsshab
P_b(\xi_2))+P_b(P_a(\xi_1)\qsshab \xi_2),
\mlabel{eq:oper}
\end{equation} where
$\xi_1$ is in the domain of $P_a$ and $\xi_2$ is in the domain of
$P_b$. Now we prove that $$\Theta(\xi_1\qsshab \xi_2)
=\Theta(\xi_1)\Theta(\xi_2)$$ for $\xi_1,\xi_2$ in the free monoid  $M(\zsg{\RR_{+}})$ generated by $\zsg{\RR_+}$. This is done
by induction on $|\xi_1|+|\xi_2|$. Here
$$|\xi_1|=\left\{
\begin{array}{ll} 1 & \text{ if } \xi_1=1,
\\
r_1 + \cdots + r_\ell, & \text{ if } \xi_1=\spair{r_1,\cdots,
r_\ell}{a_1,\cdots, a_\ell}.
 \end{array}\right.$$ If $|\xi_1|=0$ or $|\xi_2|=0$, then there is
nothing to prove. So we assume that $|\xi_1|\geq 1$ and $|\xi_2|\geq
1$. Then we can write $\xi_1=P_a(\xi'_1)$ for some $a\in \RR_{\geq 0}$
and $\xi'_1\in M(\zsg{\RR_+})$. Similarly we can write
$\xi_2=P_b(\xi'_2)$ for some $b\in \RR_{\geq 0}$ and $\xi'_2\in
M(\zsg{\RR_+})$. Then
$$\begin{aligned}
\Theta(\xi_1\qsshab \xi_2) &= \Theta(P_a(\xi'_1)\qsshab P_b(\xi'_2))
\\
&= \Theta(P_{a}( \xi'_1\qsshab P_b(\xi'_2)) ) +\Theta(
P_b(P_a(\xi'_1)\qsshab \xi'_2)) \qquad (\text{by
Eq.~(\mref{eq:oper})}
\\
&=I_a(\Theta(\xi'_1\qsshab P_b(\xi'_2)))+ I_b(\Theta
(P_a(\xi'_1)\qsshab \xi'_2)) \qquad (\text{by Eq.~(\mref{eq:comm})})
\\
&=I_a(\Theta(\xi'_1)\Theta(P_b(\xi'_2)))+ I_b(\Theta
(P_a(\xi'_1))\Theta(\xi'_2)) \qquad (\text{by induction assumption})
\\
&=I_a(\Theta(\xi'_1)I_b(\Theta(\xi'_2)))+ I_b(I_a
(\Theta(\xi'_1))\Theta(\xi'_2)) \qquad (\text{by
Eq.~(\mref{eq:comm})}) \\
&= I_a(\Theta(\xi'_1))I_b(\Theta(\xi'_2))  \qquad (\text{by Eq.~(\mref{eq:intI})}) \\
&= \Theta(P_a(\xi'_1))\Theta(P_b(\xi'_2)) \qquad (\text{by
Eq.~(\mref{eq:comm})}).
\end{aligned}$$ This completes the induction.
\end{proof}
Taking $t=0$ in Proposition~\mref{pp:realfun}, we obtain
\begin{coro} The $\RR$-linear map
$$\Theta: \calh(\zsg{\RR_+})\rightarrow \RR , \quad \spair{s_1,\cdots, s_k}{b_1,\cdots, b_k}
\mapsto \frakf\pfpair{s_1,\cdots, s_k}{b_1,\cdots, b_k}, \ 1 \mapsto
1
$$ is an $\RR$-algebra homomorphism. \mlabel{co:real}
\end{coro}

Based on this corollary we can now prove Theorem \mref{thm:pf}.
\medskip

\noindent{\it Proof of Theorem \mref{thm:pf}.} Let
$\spair{\vec{r}}{\vec{v}}\in \zsg{\ug}^k$ and
$\spair{\vec{s}}{\vec{u}}\in\zsg{\ug}^\ell$.
We have to prove the equation
\begin{equation}
\calf(\spair{\vec{r}}{\vec{v}})\calf(\spair{\vec{s}}{\vec{u}})=
\calf(\spair{\vec{r}}{\vec{v}}\qsshab\spair{\vec{s}}{\vec{u}}).
\mlabel{eq:al}
\end{equation}
Both sides of this equation are rational functions in $U$.
Since the zero set of a nonzero rational function does not contain any non-empty open subset in $\RR^{k+\ell}$ while, by Corollary~\mref{co:real}, the above equation holds when the variables $\vec{u}$ and $\vec{v}$ take values in
$\RR_+$, the equation has been proved. \qed

\section{Product formula of \MZV fractions}
\mlabel{sec:expf}

We now apply Theorem~\mref{thm:pf} and the explicit shuffle product formula obtained in~\mcite{GX2} to give an explicit product formula of \MZV fractions. We will also give some examples.

We need to recall some notations
to give this formula. For positive integers $k$ and $\ell$, denote
$[k]=\{1,\cdots,k\}$ and $[k+1,k+\ell]=\{k+1,\cdots,k+\ell\}.$
Define
\begin{equation}
\indI_{k,\ell}=\left \{(\varphi,\psi)\ \Big|\ \begin{array}{l}
\varphi: [k]\to [k+\ell], \psi: [\ell]\to [k+\ell]
\text{ are order preserving } \\
\text{ injective maps and } \im(\varphi)\cup\im (\psi)=[k+\ell] \end{array} \right
\} \mlabel{eq:ind}
\end{equation}
Let $\vec{u}\in\ug^k$, $\vec{v}\in\ug^\ell$ and
$(\varphi,\psi)\in\indI_{k,\ell}$. We define
$\vec{u}\ssha_{(\varphi,\psi)}\vec{v}$ to be the vector whose $i$th
component is
\begin{equation}
(\vec{u}\ssha_{(\varphi,\psi)} \vec{v})_i :=\left\{\begin{array}{ll}
u_j & \text{if } i=\varphi(j), \\ v_j & \text{if }
i=\psi(j),\end{array}\right. \quad 1\leq i\leq k+\ell.
\mlabel{eq:mulind}
\end{equation}

Let
$\vec{\rind}=(\rind_1,\cdots, \rind_k)\in\ZZ_{\geq 1}^k$, $\vec{\sind}=(\sind_1,\cdots,\sind_\ell)\in \ZZ_{\geq 1}^{\ell}$ and  $\vec{\tind}=(\tind_1,\cdots, \tind_{k+\ell})\in \ZZ_{\geq 1}^{k+\ell}$
with $|\vec{\rind}|+|\vec{\sind}|=|\vec{\tind}|$. Here
$|\vec{\rind}|=\rind_1+\cdots +\rind_k$ and similarly for $|\vec{\sind}|$ and $|\vec{\tind}|$.
Denote $R_i=r_1+\cdots +r_i$ for $i\in [k]$, $S_i=s_1+\cdots +s_i$ for $i\in [\ell]$ and $T_i=t_1+\cdots+t_i$ for $i\in [k+\ell]$.
For $i\in [k+\ell]$, define
\begin{equation}
h_{(\varphi,\psi),i}=h_{(\varphi,\psi),(\vec{\rind},\vec{\sind}),i}
=
       \left\{
              \begin{array}{ll} \rind_{j} & \text{ if } i=\varphi(j)
                                 \\
                                \sind_{j} & \text{ if } i=\psi(j)
              \end{array}
              \right.
              = r_{\varphi^{-1}(i)}s_{\psi^{-1}(i)},
\mlabel{eq:h0}
\end{equation}
with the convention that $r_\emptyset =s_\emptyset =1.$

With these notations, we define
\begin{equation}
c_{\vec{\rind},\vec{\sind}}^{\vec{\tind},(\varphi,\psi)}(i) =\left\{
\begin{array}{ll}
\binc{\tind_i-1}{h_{(\varphi,\psi),i}-1} &
    \text{if } i=1  \text{ or } \text{if }
i-1,i \in \im(\varphi) \text{ or if }
i-1,i \in \im(\psi),
\vspace{.2cm}
\\ \vspace{.2cm}
\begin{array}{l}
\binc{\tind_i-1} {T_i-R_{|\varphi^{-1}([i])|}-S_{|\psi^{-1}([i])|}}\\
= \binc{\tind_{i}-1}{\sum\limits_{j=1}^{i} \tind_j
-\sum\limits_{j=1}^{i} h_{(\varphi,\psi),j}}
\end{array} & \text{ otherwise}.
\end{array}
\right. \mlabel{eq:coef-re-def1}
\end{equation}

The following theorem is proved in~\mcite{GX2}.
\begin{theorem}
{\bf (\cite[Theorem 2.1]{GX2}}
Let $\ug$ be a countably infinite set and let $\calh\bsh(\zsg{\ug})=(\calh(\zsg{\ug}), \qsshab)$ be as defined by Eq.~(\mref{eq:pfsh}).
Then for $\spair{\vec{\rind}}{\vec{u}}\in \zsg{\ug}^k$ and
$\spair{\vec{\sind}}{\vec{v}}\in\zsg{\ug}^\ell$ in $\calh\bsh(\zsg{\ug})$, we have
\begin{equation}
  \spair{\vec{\rind}}{\vec{u}}\qsshab \spair{\vec{\sind}}{\vec{v}}
  =
\sum_{\tiny\begin{array}{c} (\varphi,\psi)\in \indI_{k,\ell}\\
\vec{\tind}\in \ZZ_{\geq 1}^{k+\ell},
|\vec{\tind}|=|\vec{\rind}|+|\vec{\sind}| \end{array}}
\bigg(\prod_{i=1}^{k+\ell}c_{\vec{\rind},\vec{\sind}}^{\vec{\tind},
(\varphi,\psi)}(i)\bigg)
\spair{\vec{\tind}}{\vec{u}\ssha_{(\varphi,\psi)}\vec{v}}, \mlabel{eqn:maincoef}
\end{equation}
where $c_{\vec{\rind},\vec{\sind}}^{\vec{\tind}, (\varphi,\psi)}(i)$
is given in Eq.~(\mref{eq:coef-re-def1}) and
$\vec{u}\ssha_{(\varphi,\psi)}\vec{v}$ is given in
Eq.~(\mref{eq:mulind}). \mlabel{thm:qshsh}
\end{theorem}
Then by Corollary~\mref{co:pfshf}, we have
\begin{theorem} With notations as in Theorem~\mref{thm:qshsh}, we have
\begin{equation}
  \frakf \pfpair{\vec{\rind}}{\vec{u}}
  \frakf \spair{\vec{\sind}}{\vec{v}})
  =
\sum_{\tiny\begin{array}{c} (\varphi,\psi)\in \indI_{k,\ell}\\
\vec{\tind}\in \ZZ_{\geq 1}^{k+\ell},
|\vec{\tind}|=|\vec{\rind}|+|\vec{\sind}| \end{array}}
\bigg(\prod_{i=1}^{k+\ell}c_{\vec{\rind},\vec{\sind}}^{\vec{\tind},
(\varphi,\psi)}(i)\bigg)
\frakf \pfpair{\vec{\tind}}{\vec{u} \ssha_{(\varphi,\psi)}\vec{v}}.
\mlabel{eq:pfprod}
\end{equation}
\mlabel{thm:pfprod}
\end{theorem}
Assume $r_1,s_1\geq 2$. Taking the sum $\sum\limits_{u_1,\cdots,u_k\geq 1}\sum\limits_{v_1,\cdots,v_\ell\geq 1}$ on
both sides of Eq.~(\mref{eq:pfprod}), we obtain the generalization of Euler's decomposition formula of two \MZVs in~\cite[Corollary~2.5]{GX2}.

We give some examples of Theorem~\mref{thm:pfprod}. We will only provide details for the first example and will refer the reader to \cite[Section 2.4]{GX2} for further details on the computations.

\medskip
\noindent
{\bf 1. The case of $k=\ell=1$.} Then $\vec{\rind}=\rind_1$
and $\vec{\sind}=\sind_1$ are positive integers, and $\vec{u}=u_1$
and $\vec{v}=v_1$ are variables. Let $\vec{\tind}=(\tind_1,
\tind_2)\in\ZZ_{\geq 1}^2$ with $\tind_1+\tind_2=\rind_1+\sind_1$.
If $(\varphi,\psi)\in\indI_{1,1}$, then either $\varphi(1)=1$ and
$\psi(1)=2$, or $\psi(1)=1$ and $\varphi(1)=2$.  If $\varphi(1)=1$
and $\psi(1)=2$, then as in \cite[Section 2.4]{GX2}, we obtain
$$ c_{\rind_1,\sind_1}^{\vec{\tind}, (\varphi,\psi)}=
c_{\rind_1,\sind_1}^{\vec{\tind}, (\varphi,\psi)}(1)\,
c_{\rind_1,\sind_1}^{\vec{\tind}, (\varphi,\psi)}(2)
=\binc{\tind_1-1}{\rind_1-1}.$$
By Eq.~(\mref{eq:mulind}), we have
$$
\vec{u}\ssha_{(\varphi,\psi)} \vec{v} =(u_1,v_1).
$$

If $\psi(1)=1$ and $\varphi(1)=2$, then we similarly obtain
$$ c_{\rind_1,\sind_1}^{\vec{\tind}, (\varphi,\psi)}=
 c_{\rind_1,\sind_1}^{\vec{\tind}, (\varphi,\psi)}(1)\, c_{\rind_1,\sind_1}^{\vec{\tind}, (\varphi,\psi)}(2)=
\binc{\tind_1-1}{\sind_1-1}.$$ By Eq.~(\mref{eq:mulind}), we have $
\vec{u}\ssha_{(\varphi,\psi)} \vec{v} =(v_1,u_1). $
Therefore,
$$
\frakf\pfpair{\rind_1}{u_1} \frakf\pfpair{\sind_1}{v_1} = \hspace{-.5cm}
\sum_{\tind_1,\tind_2\geq 1, \tind_1+\tind_2=\rind_1+\sind_1}
\binc{\tind_1-1}{\rind_1-1} \frakf\pfpair{\tind_1,\tind_2}{u_1,v_1} +
\hspace{-.5cm}
\sum_{\tind_1,\tind_2\geq 1, \tind_1+\tind_2=\rind_1+\sind_1}
\binc{\tind_1-1}{\sind_1-1} \frakf\pfpair{\tind_1,\tind_2}{v_1,u_1}.
$$
That is,
\begin{equation}
\frac{1}{u_1^{r_1}}\, \frac{1}{v_1^{s_1}}= \hspace{-.2cm}
\sum_{\tind_1,\tind_2\geq 1, \tind_1+\tind_2=\rind_1+\sind_1}
\binc{\tind_1-1}{\rind_1-1} \frac{1}{(u_1+v_1)^{t_1} v_1^{t_2}} +
\hspace{-.2cm}
\sum_{\tind_1,\tind_2\geq 1, \tind_1+\tind_2=\rind_1+\sind_1}
\binc{\tind_1-1}{\sind_1-1} \frac{1}{(u_1+v_1)^{t_1} u_1^{t_2}}.
\notag
%\mlabel{eq:pf1-1}
\end{equation}
This coincides with the well-known partial fraction formula~\cite[Eq.~(19)]{GKZ}
recalled in Eq.~(\mref{eq:prod2}).

\medskip
\noindent {\bf 2. The case of $r=1,s=2$.} By a similar computation of the coefficients, Eq.~(\ref{eq:pfprod})
becomes Eq.~(\mref{eq:1-2}).

\medskip
\noindent
{\bf 3. The case of $r=s=2$.}
In this case $\spair{\vec{\rind}}{\vec{w}}
=\spair{\rind_1,\rind_2}{w_1,w_2}$ and
$\spair{\vec{\sind}}{\vec{z}}=\spair{\sind_1, \sind_2}{z_1,z_2}$.
Let $\vec{\tind}=(\tind_1, \tind_2, \tind_3,\tind_4)\in\ZZ_{\geq
1}^4$ with
$\tind_1+\tind_2+\tind_3+\tind_4=\rind_1+\rind_2+\sind_1+\sind_2$.
Then there are $\binc{4}{2}=6$ choices of
$(\varphi,\psi)\in\indI_{2,2}$.
Then from Theorem~\mref{thm:pfprod}, we similarly derive
{\small \allowdisplaybreaks
\begin{eqnarray}
\lefteqn{\frakf\spair{\rind_1,\rind_2}{u_1,u_2}\, \frakf\spair{\sind_1,\sind_2}{v_1,v_2}}
\notag
\\
&=
 & \hspace{-.8cm}
\sum\limits_{\tiny \begin{array}{c}\tind_1\geq 2,\tind_2,\tind_3\geq 1
\\ \tind_1+\tind_2+\tind_3=r_1+r_2+s_1
\end{array}} \hspace{-.8cm}
 \binc{\tind_1-1}{\rind_1-1}\binc{\tind_2-1}{\rind_2-1}
\frakf\spair{t_1,t_2,t_3,s_2}{u_1,u_2,v_1,v_2}
 + \hspace{-.8cm} \sum\limits_{\tiny \begin{array}{c}\tind_1\geq 2,\tind_2,\tind_3\geq 1
\\ \tind_1+\tind_2+\tind_3=r_1+s_1+s_2
\end{array}} \hspace{-.8cm}
 \binc{\tind_1-1}{s_1-1}\binc{\tind_2-1}{s_2-1}
\frakf\spair{t_1,t_2,t_3,r_2}{v_1,v_2,u_1,u_2}
\notag \\
&&  +
\sum\limits_{\tiny \begin{array}{c}\tind_1\geq 2,\tind_2,\tind_3,\tind_4\geq 1
\\ \tind_1+\tind_2+\tind_3+\tind_4= \\
\rind_1+\rind_2+\sind_1+\sind_2
\end{array}}
\bigg [
\binc{\tind_1-1}{\rind_1-1} \binc{\tind_2-1}{t_1+t_2-r_1-s_1}
\binc{\tind_3-1}{\sind_2-\tind_4}
\frakf\spair{t_1,t_2,t_3,t_4}{u_1,v_1,u_2,v_2}
\notag\\
&&  \qquad +\binc{\tind_1-1}{ \sind_1-1}\binc{\tind_2-1}{
t_1+t_2-r_1-s_1}\binc{\tind_3-1}{\rind_2-\tind_4}
\frakf\spair{t_1,t_2,t_3,t_4}{v_1,u_1,v_2,u_2}
\notag
%\mlabel{eq:c2-2}
\\
&& \qquad
+\binc{\tind_1-1}{\rind_1-1} \binc{\tind_2-1}{t_1+t_2-r_1-s_1}
\binc{\tind_3-1}{\sind_2-1}
\frakf\spair{t_1,t_2,t_3,t_4}{u_1,v_1,v_2,u_2}
+ \binc{\tind_1-1}{ \sind_1-1}\binc{\tind_2-1}{
t_1+t_2-r_1-s_1}\binc{\tind_3-1}{
\rind_2-1}\frakf \spair{t_1,t_2,t_3,t_4}{v_1,u_1,u_2,v_2} \bigg ].
\notag
\end{eqnarray}
}

\end{document}